\documentclass[12pt]{amsart}
\usepackage{fullpage}
\usepackage{tabu}
\usepackage{amssymb} 
\usepackage{amsmath} 
\usepackage{amscd}
\usepackage{amsbsy}
\usepackage{comment, enumerate}
\usepackage[matrix,arrow]{xy}
\usepackage{mathrsfs}
\usepackage{color}
\usepackage{mathtools,caption}
\usepackage{todonotes}
\usepackage{quiver}

\definecolor{darkgreen}{rgb}{0,0.5,0}

\usepackage[
        colorlinks, citecolor=darkgreen,
        backref,
]{hyperref}

\newcommand{\F}{\mathbb{F}}

\newcommand{\Q}{\mathbb{Q}}

\newcommand{\Z}{\mathbb{Z}}

\newcommand{\rhobar}{{\bar{\rho}}}

\newcommand{\p}{\mathfrak{p}}

\DeclareMathOperator{\Gal}{Gal}

\DeclareMathOperator{\tr}{tr}

\DeclareMathOperator{\rad}{rad}
\DeclareMathOperator{\Li}{Li}
\DeclareMathOperator{\rank}{rank}

\newcommand{\Mod}[1]{\ (\mathrm{mod}\ #1)}

\newcommand{\GL}{\operatorname{GL}}

\newcommand{\surjects}{\twoheadrightarrow}
\numberwithin{equation}{section}

\newtheorem{theorem}{Theorem}[section]
\newtheorem{lemma}[theorem]{Lemma}
\newtheorem{corollary}[theorem]{Corollary}
\newtheorem{proposition}[theorem]{Proposition}

\theoremstyle{definition}
\newtheorem{definition}[theorem]{Definition}

\theoremstyle{remark}
\newtheorem{remark}[theorem]{Remark}

\begin{document}

\title[Improved bounds Serre's open image theorem]{Improved bounds for Serre's open image theorem}

\date{\today}


\author{Imin Chen}

\address{Department of Mathematics, Simon Fraser University\\
Burnaby, BC V5A 1S6, Canada.} 
\email{ichen@sfu.ca}

\author{Joshua Swidinsky}

\address{Department of Mathematics, Simon Fraser University\\
Burnaby, BC V5A 1S6, Canada.}
\email{joshua\_swidinsky@sfu.ca}

\begin{abstract}
Let $E$ be an elliptic curve over the rationals which does not have complex multiplication. Serre showed that the adelic representation attached to $E/\Q$ has open image, and in particular there is a minimal natural number $C_E$ such that the mod $\ell$ representation $\rhobar_{E,\ell}$ is surjective for any prime $\ell > C_E$. Assuming the Generalized Riemann Hypothesis, Mayle-Wang gave explicit bounds for $C_E$ which are logarithmic in the conductor of $E$ and have explicit constants. The method is based on using effective forms of the Chebotarev density theorem together with the Faltings-Serre method, in particular, using the `deviation group' of the $2$-adic representations attached to two elliptic curves. 

By considering quotients of the deviation group and a characterization of the images of the $2$-adic representation $\rho_{E,2}$ by Rouse and Zureick-Brown, we show in this paper how to further reduce the constants in Mayle-Wang's results. Another result of independent interest are improved effective isogeny theorems for elliptic curves over the rationals.
\end{abstract}

\thanks{}

\maketitle

{
\hypersetup{linkcolor=black}
\setcounter{tocdepth}{1}
\tableofcontents
}

\section{Introduction}

Let $E$ be an elliptic curve over $\Q$ without complex multiplication. Serre showed in \cite{serregalois} that the adelic representation attached to $E/\Q$ has open image, in particular, there is a minimal natural number $C_E$ such that the mod $\ell$ representation $\rhobar_{E,\ell}$ is surjective for any prime $\ell > C_E$. 

In determining effective bounds on $C_E$, one typically uses effective versions of the Chebotarev density theorem under the assumption of the Generalized Riemann Hypothesis (GRH) as was first done by Serre. The bounds on $C_E$ usually depend on the radical $\rad(N_E)$ of the conductor $N_E$ of $E$ over $\Q$. In Serre's original treatment \cite{serrechebotarev}, the following theorem was shown. By GRH, we mean the conjecture which applies to the Artin $L$-functions of Galois extensions $L/K$; unless otherwise stated, $K = \Q$.
\begin{theorem}\cite[Theorem 21]{serrechebotarev}\label{thm:Serre_elliptic_curve_prime}
Assume GRH. Let $E$ and $E'$ be two elliptic curves defined over $\Q$. Suppose that $E$ and $E'$ are not $\Q$-isogenous. Then there exists a prime $p$ of good reduction for $E$ and $E'$ such that $a_p(E) \neq a_p(E')$ and satisfying the inequality
\begin{equation}\label{eq:elliptic_curve_prime_bound} 
p \leq C_1 (\log \rad(N_E N_{E'}))^2 (\log \log \rad(2 N_E N_{E'}))^{12}, 
\end{equation}
where $C_1$ is an absolute constant.
\end{theorem}
Based on the method used, the constant $C_1$ here is unfortunately rather large. Recent work of Mayle-Wang \cite{Mayle-Wang} has given an explicit result on the smallest prime which achieves $a_p(E) \neq a_p(E')$. The constants are quite small, and like Serre, depend on only knowledge of the primes of bad reduction of the two elliptic curves $E$ and $E'$. 

\begin{theorem}\cite[Theorem 2]{Mayle-Wang}
\label{thm:mayle_wang_result}
Assume GRH. Let $E$ and $E'$ be two elliptic curves over $\Q$ without complex multiplication. Suppose $E$ and $E'$ are not $\Q$-isogenous. Then there exists a prime $p$ of good reduction for $E$ and $E'$ such that $a_p(E) \neq a_p(E')$ and satisfying the inequality
\begin{equation}
\label{eq:mayle_wang_final_result}
p \leq (482\log\rad(2 N_E N_{E'}) + 2880)^2,
\end{equation}
where $N_E$ and $N_{E'}$ denote the conductors of $E$ and $E'$, respectively.
\end{theorem}
The method is based on using effective forms of the Chebotarev density theorem together with the Faltings-Serre method, in particular, by studying the `deviation group' $\delta(G)$ of the $2$-adic representations attached to two elliptic curves. 

In our work, we explain how to replace $\delta(G)$ with smaller quotients in Mayle-Wang's original arguments. Using these smaller quotients allows us to prove an improved effective isogeny theorem for elliptic curves over $\Q$ with a certain condition on the mod $2$ representations.
\begin{theorem}\label{thm:mayle_wang_improvement}
    Assume GRH. Let $E$ and $E'$ be two elliptic curves over $\Q$. Suppose $E$ and $E'$ are not $\Q$-isogenous. Assume the mod 2 representations $\bar{\rho}_{E,2}$ and $\bar{\rho}_{E',2}$ are not isomorphic, or if they are isomorphic that they are absolutely irreducible. Then there exists a prime $p$ of good reduction for $E$ and $E'$ such that $a_p(E) \neq a_p(E')$ and satisfying the inequality 
    \begin{equation}
    \label{eq:mayle_wang_improvement_result} 
        p \leq (124 \log\rad(2 N_E N_{E'}) +  561)^2.
    \end{equation}
\end{theorem}

\begin{remark}
    Mayle-Wang, in Theorem \ref{thm:mayle_wang_result}, include a hypothesis that the elliptic curves $E$ and $E'$ be without complex multiplication; in Proposition \ref{prop:mayle_wang_generalization} and Theorem \ref{thm:mayle_wang_improvement}, we have dropped this assumption. All we require here is the existence of a prime of good reduction such that $a_p(E) \neq a_p(E')$, which is satisfied once we assume the two elliptic curves are not $\Q$-isogenous, a consequence of Faltings' Theorem \cite{faltings} (see translation in \cite{silverman_cornell}).
\end{remark}

We also prove another improved effective isogeny theorem which applies for elliptic curves over $\Q$ which are quadratic twists of each other and do not have complex multiplication.
\begin{theorem}\label{thm:mayle_wang_twist}
    Assume GRH. Let $E$ and $E'$ be two elliptic curves over $\Q$ which are quadratic twists of each other and do not have complex multiplication. Suppose $E$ and $E'$ are not $\Q$-isogenous. Then there exists a prime $p$ of good reduction for $E$ and $E'$ such that $a_p(E) \neq a_p(E')$ and satisfying the inequality 
    \begin{equation}\label{eq:mayle_wang_improvement_result_twist} 
        p \leq (223 \log\rad(2N_E N_{E'}) + 1127)^2.
    \end{equation}
\end{theorem}
This version requires the results of Rouse and Zureick-Brown \cite{Rouse-Zureick-Brown} which characterizes the images of the $2$-adic representations attached to an elliptic curve over $\Q$. 

Using Theorem~\ref{thm:mayle_wang_result}, Mayle-Wang \cite[Theorem 1]{Mayle-Wang} prove the following bound for Serre's open image theorem.
\begin{theorem}\label{thm:mayle_wang_surjective}
    Assume GRH. Let $E$ an elliptic curve over $\Q$ without complex multiplication.
    Then \[ C_E \le 964 \log \rad(2 N_E) + 5760.\]
\end{theorem}

A consequence of our improved effective isogeny Theorems \ref{thm:mayle_wang_improvement} and \ref{thm:mayle_wang_twist} is an improvement in the constants for Serre's open image theorem.
\begin{theorem}\label{thm:mayle_wang_surjective_result}
    Assume GRH. Let $E$ an elliptic curve over $\Q$ without complex multiplication.
    Then \[C_E \le 446 \log \rad(2 N_E) + 2254.\]
\end{theorem}

The {\tt Magma} computational algebra system \cite{magma} was used for verifying assertions in this paper. The electronic resources are available from 
\begin{center}
\url{https://github.com/ichensfuca/ChenSwidinsky}.
\end{center}

\section{Explicit forms of the Chebotarev density theorem}

Let $L/K$ be a finite Galois extension with Galois group $G$. Define the counting function $\pi_C(x, L/K)$, for a conjugacy class $C$ of the Galois group $G$ of $L/K$, to be the function
\[ \pi_C(x, L/K) = \left|\left\{ \p \mid \p \text{ unramified in $K$}, \left( \frac{L/K}{\p} \right) = C, \, \text{N}_{K/\Q}(\p) \leq x \right\}\right|. \]
\begin{theorem}[Chebotarev Density Theorem]\label{thm:chebotarev_analytic_version}
    Let $\pi_C(x,L/K)$ be as above. Then,
    \[ \pi_C(x,L/K) \sim \frac{|C|}{|G|} \Li(x). \]
\end{theorem}

Effective versions of Chebotarev's Density Theorem exist as well and we shall be applying results in which the constants are explicitly computable in terms of the discriminants of $L$ and $K$, as well as the degree of each extension over $\Q$. The first of these was given by Lagarias and Odlyzko \cite{Lagarias_Odlyzko}. Their result relies on the validity of GRH.

We now state the first explicit form of Theorem \ref{thm:chebotarev_analytic_version}.
\begin{theorem}\cite[Theorem 1.1]{Lagarias_Odlyzko}
    There exists an effectively computable positive absolute constant $c_1$ such that if GRH holds for the Dedekind zeta function of $L$, then for every $x \geq 2$, 
    \[ |\pi_C(x,L/K) - \frac{|C|}{|G|}\Li(x)| \leq c_1 \left( \frac{|C|}{|G|} x^{1/2} \log(|d_L|x^{n_L}) + \log |d_L| \right). \]
\end{theorem}
An important corollary, one that we shall make use of, is finding an $x_0$ such that $\pi_C(x_0,L/K) > 0$.
\begin{corollary}\cite[Corollary 1.2]{Lagarias_Odlyzko}
    There exists an effectively computable positive absolute constant $c_2$ such that if GRH holds for the Artin $L$-functions of $L/\Q$, $L \neq \Q$, then for every conjugacy class $C$ of $G$ there exists an unramified prime ideal $\p$ of $K$ such that \[ \left( \frac{L/K}{\p} \right) = C \] and 
    \[ \text{N}_{K/\Q} (\p) \leq c_2(\log |d_L|)^2(\log\log |d_L|)^4. \]
    If $L = \Q$, then $\p = (2)$ is a solution.
\end{corollary}
The above is a non-nullity result about $\pi_C(x,L/K)$; it asserts the size of $x$ we must take to ensure that $\pi_C(x,L/K)$ is nonzero, that is, there is some unramified prime ideal $\p$ of $K$ whose Artin symbol hits $C$ and with norm smaller than $x$.
\begin{theorem} \cite[Th\'eor\`eme 4]{Oesterle}
There exists an effectively computable positive absolute constant $c_3$ such that if GRH and Artin's Conjecture (AC) hold for the Artin L-functions of $L/\Q$, $L \neq \Q$, then for every conjugacy class $C$ of $G$, we have that 
    \[ \pi_C(x, L/K) \geq 1 \]
    for all $x \geq 2$ such that $x \geq c_3(\log |d_L|)^2$.
\end{theorem}
Oesterl\'e \cite[Th\'eor\`eme 4]{Oesterle} finds that $c_3 = 70$, although his proof was seemingly never published. An improvement to Lagarias and Odlyzko is given by Bach-Sorenson \cite{bach-sorenson-chebotarev}:
\begin{theorem}\cite[Theorem 5.1]{bach-sorenson-chebotarev}\label{thm:bach-sorenson}
    Assume GRH. Let $K/\Q$ be a Galois extension of number fields, with $K \neq \Q$. Let $d_K$ denote the discriminant of $K$. Let $n_K$ denote the degree of $K$. Let $C \subseteq \Gal(K/\Q)$ be a nonempty subset closed under conjugation. Then, there is a prime $p$ of $\Q$ unramified in $K$ with \[ \left( \frac{K/\Q}{p} \right) \subseteq C, \] satisfying
    \[ p \leq (a\log |d_K| + bn_K + c)^2 \]
    for some triple $(a,b,c)$ taken from \cite[Table 3]{bach-sorenson-chebotarev} according to the quantities $\log |d_K|$ and $n_K$. We may take $a = 4$, $b=2.5$, and $c=5$ to cover all cases of $\log |d_K|$ and $n_K = [K:\Q]$.
\end{theorem}

\begin{comment}
\begin{remark}
    We note the values $a = 4$, $b=2.5$, and $c=5$ apply for all values of $n_K = [K:\Q]$ and $\log |d_K|$. For particular ranges of $\log |d_K|$ and $[K:\Q]$, we may take refined constants coming from \cite[Table 3]{bach-sorenson-chebotarev}.
\end{remark}
\end{comment}

A corollary to the above is given in Mayle-Wang \cite[Corollary 6]{Mayle-Wang} when we need to pick the prime $p$ to be coprime to a given positive integer $m$.

\begin{corollary}\cite[Corollary 6]{Mayle-Wang}\label{cor:bach-sorenson}
    Assume GRH. Let $K/\Q$ be a Galois extension of number fields, with $K \neq \Q$. Let $m$ be a positive integer, and set $\tilde{K} = K(\sqrt{m})$. Denote $d_{\tilde{K}}$ to be the absolute value of the discriminant of $\tilde{K}$. Let $n_{\tilde{K}}$ denote the degree of $\tilde{K}$. Let $C \subseteq \Gal(K/\Q)$ be a nonempty subset that is closed under conjugation. Then there exists a prime number $p$ not dividing $m$ that is unramified in $K/\Q$ with $\left(\frac{K/\Q}{p}\right) \subseteq C$ and satisfying
    \[ p \leq (\tilde{a}\log |d_{\tilde{K}}| + \tilde{b}n_{\tilde{K}} + \tilde{c})^2, \]
    for some triple $(\tilde{a},\tilde{b},\tilde{c})$ taken from \cite[Table 3]{bach-sorenson-chebotarev} according to the quantities $\log |d_{\tilde{K}}|$ and $n_{\tilde{K}}$. We may take $\tilde{a} = 4$, $\tilde{b} = 2.5$, and $\tilde{c} = 5$ to cover all cases of $\log |d_{\tilde{K}}|$ and $n_{\tilde{K}} = [\tilde{K}:\Q]$.
\end{corollary}

For a fixed $n_K$, we say a triple $(a,b,c)$ is bigger than a triple $(a',b',c')$ (resp.\ a triple $(a',b',c')$ is smaller than a triple $(a,b,c)$) if \[ (a' \log |d_K| + b' n_K + c')^2 \leq (a \log |d_K| + b n_K + c)^2 \] for all values of $\log |d_K|$ in a row of \cite[Table 3]{bach-sorenson-chebotarev}.

We give our own version of Theorem \ref{thm:bach-sorenson} and Corollary \ref{cor:bach-sorenson}. The idea is to collapse \cite[Table 3]{bach-sorenson-chebotarev} into a 1-dimensional table, removing the condition on $\log |d_{\tilde{K}}|$ so that each triple is valid for a range of $n_{\tilde{K}}$. We do this by picking a ``pivot'' triple for each column, for which all triples appearing before the pivot are absorbed into a special constant $p_0(n_{\tilde{K}})$, and all triples appearing after are checked to be smaller than the pivot triple. 

The pivot triple in each column of \cite[Table 3]{bach-sorenson-chebotarev} is chosen to be the first one so that \eqref{program-ineq} holds.

\begin{table}
    \centering
    \begin{tabular}{c|c}
        $n_{\tilde{K}}$ & $(\bar{a}, \bar{b}, \bar{c})$ \\
        \hline
        2 & $(1.446, 0.23, 6.8)$  \\
        3-4 & $(1.527, 0.17, 6.4)$  \\
        5-9 & $(1.629, 0.11, 6.1)$  \\
        10-14 & $(1.667, 0.09, 6.0)$ \\
        15-49 & $(1.745, 0.04, 5.8)$  \\
        50-128 & $(1.755, 0, 5.7)$  \\
    \end{tabular}
    \caption{Triples $(\bar{a},\bar{b},\bar{c})$ as appearing in Proposition \ref{cor:bach_sorenson_condense}.}
    \label{tab:bach_sorenson_gen_values}
\end{table}

\begin{proposition}\label{cor:bach_sorenson_condense}
    Assume GRH. Let $K/\Q$ be a Galois extension of number fields with $K \neq \Q$. Let $m$ be a positive integer, and set $\tilde{K} = K(\sqrt{m})$. Denote $d_{\tilde{K}}$ to be the discriminant of $\tilde{K}$. Let $n_{\tilde{K}}$ denote the degree of $\tilde{K}/\Q$. Let $C \subseteq \Gal(K/\Q)$ be a nonempty subset that is closed under conjugation. Then there exists a triple $(\bar{a}, \bar{b}, \bar{c})$ taken from Table \ref{tab:bach_sorenson_gen_values}, a special constant $p_0(n_{\tilde{K}})$, and a prime number $p$ not dividing $m$ that is unramified in $K/\Q$ with $\left(\frac{K/\Q}{p}\right) \subseteq C$ and satisfying 
    \begin{align}
      \label{constant-bound} p & \leq \max((\bar{a}\log |d_{\tilde{K}}| + \bar{b}n_{\tilde{K}} + \bar{c})^2, p_0(n_{\tilde{K}})) \\
      \label{specific-bound} & \le (\bar{a} \cdot ((n_0-1) \log \rad(d_{\tilde{K}}) + n_0 \log n_0) + \bar b \cdot n_0 + \bar c)^2, 
    \end{align}
    where $n_0 = \max(72,n_{\tilde{K}})$. If we only have an upper bound for $n_{\tilde{K}} \le n_1$, then we have to replace each of $\bar a, \bar b, \bar c$ with the maximum of their values over entries in Table \ref{tab:bach_sorenson_gen_values} with $n_{\tilde{K}} \le n_1$, respectively, and $n_0$ with $\max(72,n_1)$.
\end{proposition}

\begin{proof}
We have written a program in {\tt Magma} which verifies the required inequalities \eqref{constant-bound}. For inequality \eqref{specific-bound}, there are two parts to check. The program checks that
\begin{equation}
\label{program-ineq}
    p_0(n_{\tilde{K}}) \le (\bar{a} \cdot ((n_0-1) \log 2 + n_0 \log n_0) + \bar b \cdot n_0 + \bar c)^2,
\end{equation}
and the inequality 
\begin{equation}
    (\bar{a}\log |d_{\tilde{K}}| + \bar{b}n_{\tilde{K}} + \bar{c})^2 \le (\bar{a} \cdot ((n_0-1) \log \rad(d_{\tilde{K}}) + n_0 \log n_0) + \bar b \cdot n_0 + \bar c)^2
\end{equation}
follows from Lemma~\ref{lem:Mayle-Wang_discriminant_bound}.
\end{proof}

\begin{lemma}\cite[Lemma 7]{Mayle-Wang}\label{lem:Mayle-Wang_discriminant_bound}
    If $K/\Q$ is a nontrivial finite Galois extension, then
    \[ \left( \frac{1}{2} \log 3 \right)[K:\Q] \leq \log |d_K| \leq ([K:\Q]-1)\log \rad(d_K) + [K:\Q]\log([K:\Q]), \]
    where $d_K$ is the absolute value of the discriminant of $K$.
\end{lemma}

\section{The deviation group $\delta(G)$}\label{sec:deviation_group}
In this section, we wish to construct a finite group, called the \textit{deviation group}, denoted $\delta(G)$, from which we can find a finite subset that will determine if the two representations are isomorphic or not.

Our treatment of the deviation group will follow the exposition given in Ignasi's thesis \cite{ignasi}. We note that Ignasi's exposition is, itself, taken from Ch\^{e}nevert's thesis \cite{chenevert}, whose work follows the work of Serre \cite{serre-method} (the propositions and lemmas which appear here, with the exception of Lemma \ref{lem:determinant}, can also be found in \cite[Chapter 5]{chenevert}). 

Let $G$ be a group, and $L$ be a finite extension of $\Q_\ell$, for $\ell$ prime, with ring of integers $\mathcal{O}_\lambda$, maximal ideal $\lambda$, and residue field $k = \mathcal{O}_\lambda/\lambda\mathcal{O}_\lambda$. We let $\pi$ be a uniformizer, so $\lambda = \pi\mathcal{O}_\lambda$. Let $\rho_1, \rho_2: G \to \GL_n(\mathcal{O}_\lambda)$ be two $\lambda$-adic representations. We begin by extending the map $\rho_1 \times \rho_2 : G \to \GL_n(\mathcal{O}_\lambda) \times \GL_n(\mathcal{O}_\lambda)$ from $G$ to the group ring $\mathcal{O}_\lambda[G]$. 

%Recall that the group ring $\mathcal{O}_\lambda[G]$ is a $\mathcal{O}_\lambda$-module with a basis being the elements of $G$; explicitly, it can be written as
%\[ \mathcal{O}_\lambda[G] = \left\{ \sum a_i g_i \mid a_i \in \mathcal{O}_\lambda, g_i \in G, \text{ with only finitely many $a_i$ nonzero} \right\}. \]

%, where $M_n(\mathcal{O}_\lambda)$ is defined to be the set of all matrices with entries in $\mathcal{O}_\lambda$,

We define the map $\rho: \mathcal{O}_\lambda[G] \to M_n(\mathcal{O}_\lambda) \oplus M_n(\mathcal{O}_\lambda)$ to be
\[ \rho\left( \sum a_ig_i \right) = \left( \sum a_i \rho_1(g_i), \sum a_i \rho_2(g_i) \right). \]
%Note that the image need not be contained in $\GL_n(\mathcal{O}_\lambda) \times \GL_n(\mathcal{O}_\lambda)$, as, in general, $\GL_n(\mathcal{O}_\lambda)$ need not be closed under taking $\mathcal{O}_\lambda$-linear combinations.

Let $M$ be the full image of $\rho$ inside $M_n(\mathcal{O}_\lambda) \oplus M_n(\mathcal{O}_\lambda)$, and consider the composition map $\delta: G \xrightarrow{\rho} M^\times \to (M/\lambda M)^\times$. 
\begin{definition}\cite[Definition 2.1.1]{ignasi}\label{def:deviation_group}
    The image $\delta(G)$ of $G$ inside $(M/\lambda M)^\times$ is called the \textit{deviation group} of the pair of representations $\rho_1,\rho_2$.
\end{definition}

\begin{remark}
    Since $M$ is a subalgebra of $R = M_n(\mathcal{O}_\lambda) \times M_n(\mathcal{O}_\lambda)$, it might be tempting to think $\delta(G)$ is a subgroup of $(R/\lambda R)^\times = \GL_2(k) \times \GL_2(k)$ but this may not be the case. See the remark after \cite[Definition 2.1.1]{ignasi}.
\end{remark}

The deviation group turns out to be finite, as described by the following proposition. 
\begin{proposition}\cite[Proposition 2.1.2]{ignasi}
    The group $\delta(G)$ is finite, and in particular we have $|\delta(G)| \leq |k|^{2n^2}$.
\end{proposition}
\begin{proof}
    $M$ is a submodule of the free $\mathcal{O}_\lambda$-module $M_n(\mathcal{O}_\lambda) \oplus M_n(\mathcal{O}_\lambda)$. Since $\mathcal{O}_\lambda$ is a local ring, $M$ is free and is of rank $r$, where $r$ satisfies
    \[ r \leq \rank(M_n(\mathcal{O}_\lambda) \oplus M_n(\mathcal{O}_\lambda)) = 2n^2. \]
    Given $M$ is a $\mathcal{O}_\lambda$-module, $M/\lambda M$ is a $k$-algebra of dimension $r$. Hence,
    \[ |\delta(G)| \leq |(M/\lambda M)^\times| \leq |k|^r \leq |k|^{2n^2} \]
    as claimed.
\end{proof}

\begin{remark}
    A similar bound on $|\delta(G)|$ is employed by Mayle-Wang in their proof of Theorem \ref{thm:mayle_wang_result}, although they do not explicitly mention the deviation group. See the proof in \cite[Theorem 2]{Mayle-Wang}.
\end{remark}

Let us turn our attention now to the practical use of $\delta(G)$, that being its ability to help us determine when two representations are isomorphic.
\begin{proposition}\cite[Proposition 2.1.3]{ignasi}\label{prop:ignasi_1}
    Let $\Sigma \subseteq G$ be a subset that surjects onto $\delta(G)$. Then, $\rho_1 \sim \rho_2$ if and only if $\tr(\rho_1(g)) = \tr(\rho_2(g))$ for all $g \in \Sigma$.
\end{proposition}

Before we introduce the next proposition, some further explanations are needed (following \cite{ignasi}). We assume now the representations $\rho_1, \rho_2 : G \to \GL_n(\mathcal{O}_\lambda)$ are not isomorphic, that is, they are not conjugate in $\GL_n(\mathcal{O}_\lambda)$, but that the residual representations $\bar{\rho}_1$ and $\bar{\rho}_2$ obtained from $\rho_1$ and $\rho_2$ by reduction modulo $\lambda$ are isomorphic. We then have an equality $\bar{\rho}_1 = P\bar{\rho}_2P^{-1}$ for some matrix $P \in M_n(k)$.

Define $\beta$ to be the largest integer such that $\rho_1$ and $\rho_2$ are conjugated modulo $\lambda^{\beta}$, that is, there is a matrix $P \in \GL_n(\mathcal{O}_\lambda)$ such that $\rho_1 \equiv P\rho_2P^{-1} \Mod{\lambda^\beta}$; we then have $\beta \geq 1$, since $\bar{\rho}_1 \cong \bar{\rho}_2$. In addition, there is an integer $\alpha \geq 1$ such that $\tr(\rho_1) \equiv \tr(\rho_2) \Mod{\lambda^\alpha}$ and $\tr(\rho_1) \not\equiv \tr(\rho_2) \Mod{\lambda^{\alpha+1}}$; in particular, $\rho_1$ and $\rho_2$ are not conjugate modulo $\lambda^{\alpha+1}$, so $\beta \leq \alpha$. Given that $\rho_1$ and $\rho_2$ are conjugate modulo $\lambda^\beta$ but not conjugate modulo $\lambda^{\beta+1}$, if we replace $\rho_2$ with a conjugate we may assume $\rho_1 \equiv \rho_2 \Mod{\lambda^\beta}$ but $\rho_1 \not\equiv \rho_2 \Mod{\lambda^{\beta+1}}$.

Hence, for any $g \in G$, we have
\begin{equation}\label{eq:theta_g_def} 
    \rho_2(g) - \rho_1(g) \equiv 0 \Mod{\lambda^\beta} \Rightarrow \rho_2(g) - \rho_1(g) = \theta_g \pi^\beta
\end{equation}
for some $\theta_g \in M_n(\mathcal{O}_\lambda)$ and $\pi$ a uniformizer of $\lambda$. Rearranging, we get an equation for $\rho_2(g)$ of the form
\begin{equation}\label{eq:rho_1_formula} 
\rho_2(g) = (I_n+\theta_g\pi^\beta\rho_1(g)^{-1})\rho_1(g),
\end{equation}
where $I_n$ is the $n \times n$ identity matrix. Let $\theta: G \to M_n(\mathcal{O}_\lambda)$ be the map $g \to \theta_g\rho_1(g)^{-1}$, and notice that \eqref{eq:rho_1_formula} becomes
\begin{equation}
    \rho_2(g) = (I_n+\pi^\beta\theta(g))\rho_1(g).
\end{equation}

\begin{proposition}\cite[Proposition 2.2.1]{ignasi}\label{prop:phi_def}
    Let $\rho_1, \rho_2 : G \to \GL_n(\mathcal{O}_\lambda)$ be representations that are not isomorphic, and suppose $\bar{\rho}_1, \bar{\rho}_2: G \to \GL_n(k)$ are isomorphic. Let $\beta$ be the largest integer such that $\rho_1$ and $\rho_2$ are conjugate modulo $\lambda^\beta$, and as above, assume $\rho_2$ has been replaced by a conjugate such that $\rho_1 \equiv \rho_2 \Mod{\lambda^\beta}$. Let 
    \begin{equation}\label{eq:deviation_group_varphi_def}
    \begin{split}
        \varphi: G & \to M_n(k) \rtimes \GL_n(k) \\
        & g \mapsto (\theta(g) \Mod{\lambda}, \rho_1(g) \Mod{\lambda})
    \end{split}
    \end{equation}
    where the semidirect product is with respect to the action of $\GL_n(k)$ on $M_n(k)$ by conjugation, that is multiplication is given by
    \[ (A,B) \cdot (C,D) = (A + BCB^{-1}, BD). \]
    Then $\varphi$ is a group homomorphism which factors through the deviation group $\delta(G)$.
\end{proposition}

\begin{remark}
    The homomorphism $\delta(G) \twoheadrightarrow \varphi(G)$ may not be injective. See \cite[Remark 2.2.2]{ignasi}.
\end{remark}

Lastly, we state a general lemma regarding determinants of matrices that we shall employ later.
\begin{lemma}\cite[Lemma 2.2.3]{ignasi}\label{lem:determinant}
    Let $R$ be a discrete valuation ring with uniformizer $\pi$, and $F$ its field of fractions. For any $A \in \GL_n(F)$,
    \[ \det(I_n+\pi A) = 1 + \pi\tr(A) + O(\pi^2). \]
\end{lemma}
It can be difficult to compute the exact size of $\delta(G)$, or find a tighter upper bound for it. We will, in the following section, work to replace $\delta(G)$ with $\varphi(G)$ in the case of $2$-adic representations. The codomain of $\varphi$ is easily understood, and hence a bound for $|\varphi(G)|$ is easily computable. This is what allows us to prove Theorem \ref{thm:mayle_wang_improvement}.

\section{The tools of Mayle-Wang}\label{sec:tools_mayle_wang}

The methodology of Mayle-Wang relies on the following proposition that is due to Serre (a proof of which can be found in \cite[Theorem 4.7]{Alina_Kiran}). The proposition which follows is a refined version of Serre's original argument due to Mayle-Wang \cite[Proposition 12]{Mayle-Wang} in which we have reworked the statement and proof to follow the work and notation done in Section \ref{sec:deviation_group}. We note that the statement is similar to that of Proposition \ref{prop:ignasi_1}: here, we show that if the representations are not isomorphic, then their traces must disagree on some finite set. While the proofs are very similar, the advantage of the following proposition is that it is in a form to which we may readily apply Chebotarev.
\begin{proposition}\cite[Proposition 12]{Mayle-Wang}\label{prop:Mayle-Wang}
Let $n$ a positive integer. Let $G$ be a group and $\rho_1, \rho_2 : G \rightarrow \GL_n(\mathcal{O}_\lambda)$ be representations, and $\delta(G)$ the deviation group of $G$ with respect to the two representations $\rho_1$ and $\rho_2$. Suppose that there exists an element $g \in G$ such that $\tr \rho_1(g) \not= \tr \rho_2(g)$. Then there exists a subset $C \subseteq \delta(G)$ for which
\begin{enumerate}
    \item the set $C$ is non-empty and closed under conjugation by $\delta(G)$, and
    \item if the image in $\delta(G)$ of an element $g \in G$ belongs to $C$, then $\tr \rho_1(g) \not= \tr \rho_2(g)$.
\end{enumerate}
\end{proposition}
\begin{proof}
    Let $R :=M_n(\mathcal{O}_\lambda) \times M_n(\mathcal{O}_\lambda)$. Let $M$ denote the $\mathcal{O}_\lambda$-subalgebra of $R$ generated by the image of $G$ under the product map 
    \[ \rho_1 \times \rho_2 : G \to \GL_n(\mathcal{O}_\lambda) \times \GL_n(\mathcal{O}_\lambda). \]
    Recall that $\delta(G)$ is the image of $G$ under $\rho_1 \times \rho_2$ in $M/\lambda M$. 
    
    Let $\alpha$ be the largest nonnegative integer such that for each $g \in G$, one has that 
    \[ \tr(\rho_1(g)) \equiv \tr(\rho_2(g)) \Mod{\lambda^\alpha}. \]
    As $M$ is a $\mathcal{O}_\lambda$-subalgebra generated by the image of $G$ under $\rho_1 \times \rho_2$, it follows that the congruence $\tr x_1 \equiv \tr x_2 \Mod{\lambda^\alpha}$ holds for each pair $(x_1,x_2) \in M$. We obtain the $\mathcal{O}_\lambda$-module homomorphism $\phi: M \to \mathcal{O}_\lambda$ given by
    \[ \phi(x_1,x_2) = \lambda^{-\alpha}(\tr(x_2) - \tr(x_1)). \]
    Since $\phi(\lambda M) \subseteq \lambda\mathcal{O}_\lambda$, we may consider the induced $\mathcal{O}_\lambda/\lambda \mathcal{O}_\lambda$-module homomorphism $\bar{\phi}:M/\lambda M \to \mathcal{O}_\lambda/\lambda \mathcal{O}_\lambda$.

    Let $C = \delta(G) \setminus \ker \bar{\phi}$ be the set of elements in $\delta(G)$ whose image under $\bar{\phi}$ in $M/\lambda M$ all are nonzero. From the definition of $\alpha$ and the linearity of the trace map, there exists $g_0 \in G$ such that 
    \[ \tr(\rho_1(g_0)) \not\equiv \tr(\rho_2(g_0)) \Mod{\lambda^{\alpha+1}}. \] 
    Note that the image of $(\rho_1 \times \rho_2)(g_0)$ in $\delta(G)$ is contained in $C$, so $C$ is nonempty. Also, $C$ is closed under conjugation since the trace map is invariant under conjugation.
    
    %Further, note that the map $\phi$ is invariant under conjugation, as
    %\begin{equation}
    %\begin{split}
    %    \phi(ax_1a^{-1},bx_2b^{-1})
    %    & = \lambda^{-\alpha}(\tr(ax_1a^{-1}) - \tr(bx_2b^{-1}) \\
    %    & = \lambda^{-\alpha}(\tr(x_1) - \tr(x_2))
    %\end{split}
    %\end{equation}
    %since the trace operation is invariant under conjugation, so the induced homomorphism obtained by the reduction of $\phi$ to $M/\lambda M$ is also invariant under conjugation. Thus, if a pair $(x_1,x_2) \in C$, then any conjugate of it will have a nonzero image in $\bar{\phi}$, so $C$ is closed under conjugation. 
    
    Finally, suppose that $g \in G$ is such that the image of $g$ in $\delta(G)$ is contained in $C$. Then, $\phi(\rho_1 \times \rho_2(g)) \not\in \lambda\mathcal{O}_\lambda$, and in particular $\tr \rho_1(g) \neq \tr \rho_2(g)$.
\end{proof}
\begin{remark}
    In the notation $\delta(G)$ and $\varphi(G)$ we suppress the dependence on the representations $\rho_1$ and $\rho_2$ as they are usually fixed in the context.
\end{remark}

We now give an analogous version of Proposition \ref{prop:Mayle-Wang} in the case where the mod 2 representations are isomorphic and absolutely irreducible. This allows us to replace $\delta(G)$ in Proposition \ref{prop:Mayle-Wang} with $\varphi(G)$ from Proposition \ref{prop:phi_def}, a set which is easier to estimate the size of. The idea to replace $\delta(G)$ with $\varphi(G)$ comes from Ch\^{e}nevert \cite[pg. 114]{chenevert}, in which he gives a remark that, in the $2$-adic case, Serre \cite{serre-method} implies that $\delta(G) \cong \varphi(G)$. However, in a conversation with Ch\^{e}nevert, Serre mentions he might not have proven the map $\delta(G) \to \varphi(G)$ was an isomorphism, but, in an unpublished letter to Tate, that the $\alpha$ in the proof of Proposition \ref{prop:ignasi_1} is equal to the $\beta$ coming from the construction of the function $\varphi$ if the residual representation is surjective. We show, in the 2-adic case, that $\alpha = \beta$, and that we can replace $\delta(G)$ in Proposition \ref{prop:Mayle-Wang} with $\varphi(G)$ and get the same conclusion, that is, there is a subset $C \subseteq \varphi(G)$ that is a conjugacy class, and if $g \in G$ is such that $\varphi(g) \in C$, then $\tr \rho_1(g) \neq \tr \rho_2(g)$.

In order to prove this special case, we require a theorem of Carayol \cite{carayol}.
\begin{theorem}\cite[Theorem 1]{carayol}\label{thm:carayol}
Let $A$ be a local ring, $R$ an $A$-algebra, and let $\rho_1, \rho_2 : R \to M_n(A)$ be two representations of $R$ of the same dimension $n$. Suppose that the residual representation $\bar{\rho} : R \otimes_A F \to M_n(F)$, where $F$ is the residue field of $A$, is absolutely irreducible. Suppose that the traces for $\rho_1$ and $\rho_2$ are the same for every $r \in R$. Then, $\rho_1$ and $\rho_2$ are isomorphic as representations, that is, there exists a matrix $Q \in \GL_n(A)$ such that $\rho_1(r) = Q\rho_2(r)Q^{-1}$ for all $r \in R$.
\end{theorem}

We now prove the special case.
\begin{proposition}\label{prop:mayle_wang_varphi}
Let $n$ be a positive integer. Let $G$ be a group and $\rho_1, \rho_2 : G \rightarrow \GL_n(\Z_2)$ be representations, and suppose their reductions $\bar{\rho}_1, \bar{\rho}_2$ are isomorphic and absolutely irreducible. Suppose that there exists an element $g \in G$ such that $\tr \rho_1(g) \not= \tr \rho_2(g)$. Then there exists a subset $C \subseteq \varphi(G)$ for which
\begin{enumerate}
    \item the set $C$ is non-empty and closed under conjugation by $\varphi(G)$, and
    \item if the image in $\varphi(G)$ of an element $g \in G$ belongs to $C$, then $\tr \rho_1(g) \not= \tr \rho_2(g)$.
\end{enumerate}
\end{proposition}
\begin{proof}
    Our setup begins, as it did, in Section \ref{sec:deviation_group}. Let $\alpha$ be the largest nonnegative integer such that for each $g \in G$, we have 
    \[ \tr(\rho_1(g)) \equiv \tr(\rho_2(g)) \Mod{2^\alpha} \quad \text{and} \quad \tr(\rho_1(g)) \not\equiv \tr(\rho_2(g)) \Mod{2^{\alpha+1}}. \]
    In addition, we let $\beta$ be the largest integer such that $\rho_1$ and $\rho_2$ are conjugated modulo $2^{\beta}$, that is, there is a matrix $P \in \GL_n(\Z_2)$ such that $\rho_1 \equiv P\rho_2P^{-1} \Mod{2^\beta}$. As demonstrated before, we have $\beta \leq \alpha$. Also, given that $\rho_1$ and $\rho_2$ are conjugate modulo $2^\beta$ but not conjugate modulo $\lambda^{\beta+1}$, if we replace $\rho_2$ with a conjugate $P\rho_2 P^{-1}$ for $P \in \GL_n(\Z_2)$, we may assume
    \begin{equation}\label{eq:equivalence_mod_beta} 
    \rho_1 \equiv P\rho_2P^{-1} \Mod{2^\beta} \text{ and } \rho_1 \not\equiv P\rho_2P^{-1} \Mod{2^{\beta+1}}.
    \end{equation}
    This implies $P\rho_2(g)P^{-1} - \rho_1(g) \equiv 0 \Mod{2^{\beta}}$ for any $g \in G$. In particular, we get $P\rho_2(g)P^{-1} - \rho_1(g) = \theta_g2^\beta$ for some $\theta_g \in M_n(\Z_2)$, which we can write as 
    \begin{equation}\label{eq:theta_g_equation}
    \theta_g = \frac{P\rho_2(g)P^{-1} - \rho_1(g)}{2^{\beta}}.
    \end{equation}
    In particular, note that
    \begin{equation}\label{eq:tr_theta_g}
        \tr(\theta_g) = 2^{-\beta}(\tr(P\rho_2(g)P^{-1}) - \tr(\rho_1(g))) = 2^{-\beta}(\tr(\rho_2(g)) - \tr(\rho_1(g)))
    \end{equation}
    by the invariance of trace under conjugation.
    
    We now show $\alpha = \beta$. Extend the maps $\rho_1,\rho_2$ to the group ring $\Z/2^\alpha\Z[G]$ by $\rho_i(\sum a_jg_j) = \sum a_j\rho_i(g_j)$, for $i = 1,2$ and $a_j \in \Z/2^\alpha\Z$ and $g_j \in G$. Then, notice that
    \begin{equation}
    \begin{split}
        \tr(\rho_1(\sum a_jg_j)) \Mod{2^\alpha}
        & \equiv \tr(\sum a_j \rho_1(g_j)) \Mod{2^{\alpha}} \\
        & \equiv \sum a_j\tr(\rho_1(g_j)) \Mod{2^{\alpha}} \\
        & \equiv \sum a_j \tr(\rho_2(g_j)) \Mod{2^{\alpha}} \\
        & \equiv \tr(\rho_2(\sum a_jg_j)) \Mod{2^\alpha}.
    \end{split}
    \end{equation}
    Since we satisfy the hypotheses of Theorem \ref{thm:carayol} with $A = \Z/2^{\alpha}\Z$ and $R = \Z/2^{\alpha}\Z[G]$, we can find a matrix $Q \in \GL_n(\Z/2^{\alpha}\Z)$ such that $\rho_1(g) \equiv Q\rho_2(g)Q^{-1} \Mod{2^{\alpha}}$ for all $g \in G$. However, $\beta$ is the largest integer such that $\rho_1$ and $\rho_2$ are conjugate modulo $2^{\beta}$, so $\alpha \leq \beta$, implying $\alpha = \beta$.
    
    Recall, from \eqref{eq:deviation_group_varphi_def}, the map $\varphi: G \to M_n(\F_2) \rtimes \GL_n(\F_2)$ is defined by
    \begin{equation}\label{eq:phi_prime_1} 
    \varphi(g) = (\theta(g) \Mod{2}, \rho_1(g) \Mod{2)}) = ([\theta_g\rho_1(g)^{-1}]_2, [\rho_1(g)]_2).
    \end{equation}
    We note our use of the notation $[N]_2$, for $N \in M_n(\Z_2)$, to denote the residue class of $N$ modulo $2$.
    
    %From \eqref{eq:equivalence_mod_beta}, replacing $\beta$ with $\alpha$, we see $\rho_1 \equiv P\rho_2P^{-1} \Mod{2^\alpha}$. 
    %With the matrix With the matrix $P \in \GL_n(\Z/2^\alpha\Z)$ coming from Theorem \ref{thm:carayol}, 
    Define the map $\phi' : M_n(\F_2) \rtimes \GL_n(\F_2) \to \F_2$ by
    \begin{equation} 
    \phi'((A,B)) = \tr(AB)
    \end{equation}
    where the product of matrices is taken to be the action of $\GL_n(\F_2)$ on $M_n(\F_2)$, and the trace is considered to be modulo $2$.
    By \eqref{eq:phi_prime_1}, notice
    \begin{equation}\label{eq:phi'(varphi)_identity}
    \begin{split} 
        \phi'(\varphi(g)) 
        & = \tr([\theta_g\rho_1(g)^{-1}]_2[\rho_1(g)]_2) \\
        %& = \tr([\theta_g\rho_2(g^{-1})]_2[P\rho_1(g)P^{-1}]_2) \\
        %& = \tr([\theta_g\rho_2(g)^{-1}]_2[\rho_2(g)]_2) \\
        %& = \tr([\theta_g]_2[\rho_1(g)^{-1}]_2[\rho_1(g)]_2) \\
        & = \tr([\theta_g\rho_1(g)^{-1}\rho_1(g)]_2) \\
        & = \tr([\theta_g]_2) \\
        %& = 2^{-\alpha}\tr((\rho_2(g) - \rho_1(g))\rho_2(g^{-1})P\rho_1(g)P^{-1}) \Mod{2} \\
        %& = 2^{-\alpha} (\tr(P\rho_1(g)P^{-1}) - \tr(\rho_1(g)\rho_2(g^{-1})P\rho_1(g)P^{-1})) \Mod{2} \\
        & = [\tr(\theta_g)]_2 \\
        & = [2^{-\alpha}(\tr(\rho_2(g)) - \tr(\rho_1(g)))]_2
    \end{split}
    \end{equation}
    where we have used \eqref{eq:tr_theta_g} above (with $\beta$ replaced with $\alpha$, since $\alpha = \beta$)  and the fact that, for a matrix $N \in M_n(\Z_2)$ with $\tr(N) = \sum_{i=1}^{n} a_{ii}$ for entries $a_{ii} \in \Z_2$ along the diagonal, we have
    \begin{equation}\label{eq:trace_property}
    \begin{split}
        \tr([N]_2) 
        & = \sum_{i = 1}^{n} [a_{ii}]_2 \\
        & = \left[ \sum_{i = 1}^{n} a_{ii} \right]_2 \\
        & = \left[\tr (N)\right]_2
    \end{split}
    \end{equation}
    which shows the final equality (noting our use of $[x]_2$ in \eqref{eq:trace_property} to denote the residue class of a $2$-adic integer $x \in \Z_2$).

    %Let $C = \varphi(G) \setminus \ker \phi'$ be the set of elements in $\varphi(G)$ which take nonzero values under $\phi'$. 
    Let $C$ be the set of elements in $\varphi(G)$ that take a nonzero value under the map $\phi'$. From the definition of $\alpha$ and the linearity of the trace map, there exists $g_0 \in G$ such that 
    \[ \tr(\rho_1(g_0)) \not\equiv \tr(\rho_2(g_0)) \Mod{2^{\alpha+1}}. \] 
    Note that the image of $g_0$ in $\varphi(G)$ is inside $C$, so $C$ is nonempty. In addition, let $\phi(h) \in \phi(G)$ for some $h \in G$; then, given $\varphi$ is a homomorphism (Proposition \ref{prop:phi_def}) and by \eqref{eq:phi'(varphi)_identity} and the invariance under conjugation of the trace map, 
    \begin{equation}
    \begin{split}
        \phi'(\varphi(h)\varphi(g)\varphi(h)^{-1}) 
        & = \phi'(\varphi(hgh^{-1})) \\
        & = [2^{-\alpha}(\tr(\rho_2(hgh^{-1})) - \tr(\rho_1(hgh^{-1})))]_2 \\
        & = [2^{-\alpha}(\tr(\rho_2(h)\rho_2(g)\rho_2(h)^{-1}) - \tr(\rho_1(h)\rho_1(g)\rho_1(h)^{-1}))]_2 \\
        & = [2^{-\alpha}(\tr(\rho_2(g)) - \tr(\rho_1(g)))]_2 \\
        & = \phi'(\varphi(g)) \\
        & \neq 0,
        %& = 2^{-\alpha}(\tr([\rho_2(hgh^{-1})]_2) - \tr([\rho_1(hgh^{-1})]_2)) \\
        %& = 2^{-\alpha}([\tr(\rho_2(h)\rho_2(g)\rho_2(h)^{-1})]_2 - [\tr(\rho_1(h)\rho_1(g)\rho_1(h)^{-1})]_2) \\
        %& = 2^{-\alpha}([\tr(\rho_2(g))]_2 - [\tr(\rho_1(g))]_2) \\
        %& = 2^{-\alpha}(\tr([\rho_2(g)]_2) - \tr([\rho_1(g)]_2)) \\
        %& = \phi'(\varphi(g)) \\
        %& \neq 0,
    \end{split}
    \end{equation} 
    so $C$ is closed under conjugation. Finally, suppose that $g \in G$ is such that the image of $g$ in $\varphi(G)$ is contained in $C$. Then, $\phi'(\varphi(g)) \not= 0$, in particular $\tr \rho_1(g) \neq \tr \rho_2(g)$.
\end{proof}

\section{Improved bounds for the effective isogeny theorem}\label{sec:mayle-wang-improvement}

%We now give a proof for Proposition \ref{prop:mayle_wang_generalization}. Then, using the improvements brought about by the deviation group and Proposition \ref{prop:phi_def}, we give an improvement of the result of Mayle-Wang.

For a prime $\ell$ and an elliptic curve $E$, we define 
\[ \Q(E[\ell^{\infty}]) = \bigcup_{k=1}^{\infty} \Q(E[\ell^k]). \]

%Recall the two-dimensional mod $\ell$ representations attached to $E$ and $E'$ are representations
%\[ \bar{\rho}_{E,\ell} : G_\Q \to \GL_2(\Z/\ell\Z) \text{ and } \bar{\rho}_{E',\ell} : G_\Q \to \GL_2(\Z/\ell\Z). \]

We begin with the proof of Proposition \ref{prop:mayle_wang_generalization}. We note that the work which follows is the same as that of Mayle-Wang \cite{Mayle-Wang}, except for our use of Proposition \ref{cor:bach_sorenson_condense} and Table \ref{tab:bach_sorenson_gen_values}.

Let $E$ and $E'$ be two elliptic curves over $\Q$ and let $A = E \times E'$. Let $G = \Gal(\Q(A[2^{\infty}])/\Q))$. We may regard the $2$-adic representations $\rho_{E,2}$ and $\rho_{E',2}$ as representations of $G$ instead of $G_\Q$ since $\Q(E[2^\infty])$ and $\Q(E'[2^\infty])$ are subfields of $\Q(A[2^{\infty}])/\Q$. 

The representation $\rho_{A,2} = \rho_{E,2} \times \rho_{E',2}$ and is a continuous homomorphism $G \rightarrow \GL_2(\Z_2) \times \GL_2(\Z_2)$ with image being $M$ as defined in the proof of Proposition~\ref{prop:Mayle-Wang}. Since $2 M$ is closed inside $M$, we see that $\delta(G)$ is a closed subgroup of $G$ and hence by the fundamental theorem of infinite Galois theory corresponds to a finite Galois extension $K/\Q$ with $K \subseteq \Q(A[2^\infty])/\Q$.

Since $K \subseteq \Q(A[2^n])$ for some $n \in \mathbb{N}$, we have that 
\begin{equation}
\label{dg-divide} |\delta(G)| = [K:\Q] \mid [\Q(A[2^n]):\Q]  \mid  |\GL_2(\Z/2^n\Z)|^2 = (6 \cdot 16^{n-1})^2,
\end{equation}
for some $n \in \mathbb{N}$ and by \cite[Proposition 12]{Mayle-Wang}, 
\begin{equation}
\label{dg-bound}
  |\delta(G)| \le 255.
\end{equation}

The set $\varphi(G)$ (defined in Proposition \ref{prop:phi_def}) is a subset of a very explicit semi-direct product, and estimating $|\varphi(G)|$ gives a smaller bound.

\begin{proposition}\label{prop:mayle_wang_generalization}
    Assume GRH. Let $E$ and $E'$ be two elliptic curves over $\Q$. Suppose $E$ and $E'$ are not $\Q$-isogenous. Let $\delta(G)$ be the deviation group of $G$ with respect to the $2$-adic representations $\rho_{E,2}$ and $\rho_{E',2}$. 

    Choose the triple $(\bar{a}, \bar{b}, \bar{c})$ from Table \ref{tab:bach_sorenson_gen_values} for $n_0 = \max(72, 2|\delta(G)|)$. Then there exists a prime $p$ of good reduction for $E$ and $E'$ such that $a_p(E) \neq a_p(E')$ such that
    \begin{equation}\label{eq:mayle_wang_generalization_result} 
    p \leq ( \bar{a}((2|\delta(G)|-1)\log\rad(2N_EN_{E'}) + 2|\delta(G)|\log(2|\delta(G)|)) + 2|\delta(G)|\bar{b} + \bar{c})^2.
    \end{equation}

    Furthermore, if $E$ and $E'$ are such that their mod 2 representations are isomorphic and absolutely irreducible, then we may replace $|\delta(G)|$ with $|\varphi(G)|$.
\end{proposition}
\begin{proof}
Let $\mathcal{O}_\lambda = \Z_2$. Let $A = E \times E'$ and apply Proposition \ref{prop:Mayle-Wang} with $\ell = 2$, $n = 2$, $G = \Gal(\Q(A[2^{\infty}])/\Q))$, and the $2$-adic representations $\rho_1 = \rho_{E,2}$ and $\rho_2 = \rho_{E',2}$. By Faltings theorem \cite{faltings}, since the two elliptic curves $E$ and $E'$ are not $\Q$-isogenous, $\rho_1$ and $\rho_2$ are not isomorphic; therefore, by Serre \cite[pg. IV-15]{serre_ell-adic_representation}, there is some prime $p$ such that $a_p(E) \neq a_p(E')$. By Proposition \ref{prop:Mayle-Wang}, there exists a conjugacy class $C \subseteq \delta(G)$ obeying the stated conclusion of this proposition. 

 Let $K$ be the subfield of $\Q(A[2^\infty])$ for which $\Gal(K/\Q) = \delta(G)$. Choosing $m = \rad(N_{E} N_{E'})$, Proposition \ref{cor:bach_sorenson_condense} produces a prime $p$ not dividing $m$ such that $\left(\frac{K/\Q}{p}\right) \subseteq C$ and 
\begin{align}\label{eq:mayle_wang_improvement_1}
 p & \leq (\bar{a} \cdot ((n_0-1) \log \rad(d_{\tilde{K}}) + n_0 \log n_0) + \bar b \cdot n_0 + \bar c)^2, 
\end{align}
where $n_0 = \max(72, 2 |\delta(G)|)$ and $\bar{a},\bar{b},\bar{c}$ are the maximum of their values over entries in Table \ref{tab:bach_sorenson_gen_values} with $n_{\tilde{K}} \le 2 |\delta(G)|$, respectively. It follows from Proposition~\ref{prop:Mayle-Wang} that
\[ \tr \rho_{E,2}(\text{Frob}_p) \neq \tr \rho_{E',2}(\text{Frob}_p), \]
and consequently $a_p(E) \neq a_p(E')$.

The abelian variety $A$ has good reduction at some prime $q$ if and only if both $E$ and $E'$ have good reduction at $q$. Thus, $K/\Q$ is unramified outside of the prime divisors of $m = N_E N_{E'}$. As $\tilde{K}$ is the compositum of $K$ and $\Q(\sqrt{m})$, the primes that ramify in $\tilde{K}$ are precisely those that ramify in $K$ or in $\Q(\sqrt{m})$. Since $\rad(d_{\Q(\sqrt{m})}) \mid \rad(2m) = \rad(2N_E N_{E'})$, and rad$(d_K) \mid $ rad$(2N_E N_{E'})$, we have that 
\begin{equation}\label{eq:mayle_wang_improvement_2} 
    \rad(d_{\tilde{K}}) = \text{rad}(d_Kd_{\Q(\sqrt{m})}) \mid \text{rad}(2N_E N_{E'}).
\end{equation}
Now, applying Lemma \ref{lem:Mayle-Wang_discriminant_bound} to \eqref{eq:mayle_wang_improvement_1} gives us
\begin{align*} 
    p 
    & \leq (\bar{a}((2|\delta(G)|-1)\log\rad(2N_E N_{E'}) + 2|\delta(G)|\log(2|\delta(G)|)) + 2|\delta(G)|\bar{b} + \bar{c})^2
\end{align*}
which matches \eqref{eq:mayle_wang_generalization_result}.

To prove the final statement, note if $\bar{\rho}_{E,2}$ and $\bar{\rho}_{E',2}$ are isomorphic and absolutely irreducible, then we instead apply Proposition \ref{prop:mayle_wang_varphi} over Proposition \ref{prop:Mayle-Wang}, in which case $\delta(G)$ is replaced with $\varphi(G)$; in particular, $|\delta(G)|$ can be replaced with $|\varphi(G)|$ in \eqref{eq:mayle_wang_generalization_result}.
\end{proof}

Now we give a proof of Theorem \ref{thm:mayle_wang_improvement}.

\begin{proof}[Proof of Theorem \ref{thm:mayle_wang_improvement}]
We split our analysis into two cases. If the mod $2$ representations are not isomorphic, then mod $2$ already distinguishes the traces. Define
\begin{equation}\label{eq:rho_ell_def}
    \bar{\rho}_2 : G_\Q \to \GL_2(\Z/2\Z) \times \GL_2(\Z/2\Z)
\end{equation}
by $\bar{\rho}_2(x) = (\bar{\rho}_{E,2}(x), \bar{\rho}_{E',2}(x))$. Let $G_2 = \bar{\rho}_2(G_\Q) \subset \GL_2(\Z/2\Z) \times \GL_2(\Z/2\Z)$ be the image of the map $\bar{\rho}_2$. Let 
\begin{equation}\label{eq:C_ell_definition}
    C_2 = \{ (s,s') \in G_2 \mid \tr(s) \neq \tr(s') \}.
\end{equation}    
    Apply Proposition~\ref{cor:bach_sorenson_condense} to the field $L = \Q(E[2],E'[2])$, whose Galois group is $G_2$, the conjugacy class $C_2$ given in \eqref{eq:C_ell_definition}, and $m = \rad(N_{E} N_{E'})$, so that we get a prime $p$ unramified in $L$ and $p \nmid m$ such that $a_p(E) \not\equiv a_p(E') \Mod{2}$ (implying $a_p(E) \neq a_p(E')$) and satisfying
    \begin{align}
    p & \leq \max((\bar{a} \log |d_{\tilde{L}}| + \bar b n_{\tilde{L}} + \bar{c})^2, p_0(n_{\tilde{L}})) \\
    \label{before-dK} & \le \left( 1.755 \left( (2 n_L-1)\log\rad(d_{\tilde{L}}) + 2 n_L\log(2 n_L) \right) + 0.23 \cdot 2 n_L + 6.8 \right).
    \end{align}
    Taking $[L:\Q] \leq |\GL_2(\F_2)|^2 = 6^2 = 36$ gives us
    \begin{equation}\label{eq:mayle_wang_improvement_max_4}
    \begin{split}
        p
        & \leq \left( 1.745 \left( (71)\log\rad(d_{\tilde{L}}) + 72 \log( 72) \right) + 72 \cdot 0.23 + 6.8 \right)^2 \\
        & \leq \left( 124 \log\rad(2N_E N_{E'}) + 561 \right)^2 \\
    \end{split}
    \end{equation}

 Next, assume that the mod 2 representations $\rho_1 = \bar{\rho}_{E,2}$ and $\rho_2 = \bar{\rho}_{E',2}$ are isomorphic and absolutely irreducible. Apply Proposition \ref{prop:mayle_wang_generalization}, and replace $\delta(G)$ with $\varphi(G)$ (since the mod 2 representations are isomorphic and absolutely irreducible) to get a prime $p$ such that $a_p(E) \neq a_p(E')$ and satisfying
    \[     
    p \leq (\bar{a}((2|\varphi(G)|-1)\log\rad(2N_E N_{E'}) + 2|\varphi(G)|\log(2|\varphi(G)|)) + 2|\varphi(G)|\bar{b} + \bar{c})^2.
    \]
    From \eqref{eq:rho_1_formula} and Lemma \ref{lem:determinant}, we have for any $g \in G$,
    \begin{equation}
    \begin{split}
        \det(\rho_2(g)) 
        & = \det((I_2+2^\beta\theta(g))\rho_1(g)) \\
        & = (1+2^{\beta}\tr(\theta(g)) + O(2^{2\beta}))\det(\rho_1(g)).
    \end{split}
    \end{equation}
    As we have $\det \rho_1 = \det \rho_2$ being the cyclotomic character, so the above can be rewritten as $0 = 2^{\beta}\tr(\theta(g))+O(2^{2\beta})$, which, after multiplying through by $2^{-\beta}$ implies
    \[ \tr(\theta) \equiv 0 \Mod{2}. \]
    In particular, the map $\varphi$ from Proposition \ref{prop:phi_def} takes values in $M_2^0(\F_2) \rtimes \GL_2(\F_2)$, where $M_2^0(\F_2)$ denotes the matrices with trace $0$ with entries in $\F_2$.
    Therefore, we have $|\varphi(G)| \leq |M_2^0(\F_2) \rtimes \GL_2(\F_2)| = 8 \cdot 6 = 48$. We find from Proposition~\ref{cor:bach_sorenson_condense} that
    \begin{equation}\label{eq:mayle_wang_improvement_max_1}
    \begin{split}
        p 
        & \leq (1.745(95\log\rad(2N_E N_{E'}) + 96\log(96)) +96 \cdot 0.23 + 6.8)^2 \\
        & \leq (166 \log\rad(2N_E N_{E'}) + 794)^2 \\
    \end{split}
    \end{equation}

We now improve the bound in \eqref{eq:mayle_wang_improvement_max_1} by using the following two results.
\begin{proposition}
  Let $G$ be a group and $\rho_1, \rho_2 : G \rightarrow \GL_2(\Z_2)$ be a group homomorphism, and suppose the mod 2 representations $\bar{\rho}_1, \bar{\rho}_2$ are isomorphic. Let $\Xi$ be the elements $g \in G$ such that the characteristic polynomials of $\rho_1(g)$ and $\rho_2(g)$ coincide. Then for $g \in \Xi$, the order of $\delta(g)$ in $\delta(G)$ is $\le 3$.
\end{proposition}
\begin{proof}
    See \cite[Proposition 5.5.6]{chenevert}.
\end{proof}

\begin{corollary}
\label{chenevert-quotient}
    Let $G$ be a group and $\rho_1, \rho_2 : G \rightarrow \GL_2(\Z_2)$ be a group homomorphism, and suppose the mod 2 representations $\bar{\rho}_1, \bar{\rho}_2$ are isomorphic. Let $\pi : \delta(G) \surjects \bar G$ be a quotient having a conjugacy class $C \subseteq \bar G$ of order $> 3$. If $g \in G$ is such that $\pi(\delta(g)) \in C$, then $g \not\in \Xi$.
\end{corollary}

From \eqref{dg-divide} and \eqref{dg-bound}, the possible sizes of $\delta(G)$ are 
\begin{equation}
\label{order-list}
  1, 2, 3, 4, 6, 8, 9, 12, 16, 18, 24, 32, 36, 48, 64, 72, 96, 128, 144, 192.
\end{equation}
As the possible orders are small, it is possible to enumerate in {\tt Magma} all isomorphism classes of groups of these sizes. The group $\delta(G)$ therefore lies in an explicit finite list which can be computed.

By Corollary~\ref{chenevert-quotient}, if $\delta(G)$ has a quotient $\bar G$ with an element of order $>3$, then we may replace $\delta(G)$ by $\bar G$. We check this in {\tt Magma} for each value of $|\delta(G)|$ and find that either $\delta(G)$ has such a quotient with strictly smaller size in the given list of \eqref{order-list} or it is in a small list of problematic groups. We list for each size $|\delta(G)|$, the {\tt Magma} labels of the isomorphism classes of the problematic groups of that order.

\begin{table}[h]
\begin{tabular}{c||c|c}
$|\delta(G)|$ & \# of isomorphism classes & problematic groups \\
\hline 
192 & 1543 & 1023, 1025, 1541\\ 
144 & 197 & none \\
128 & 2328 & 2326, 2327, 2328 \\
96 & 231 & 204 \\
72 & 50 & none \\
64 & 267 & 266, 267 \\
48 & 52 & 3, 50\\
36 & 14 & 11 \\
32 & 51 & 49, 50, 51
\end{tabular}
\caption{List of problematic groups}
\label{problem-table}
\end{table}

Consider the homomorphism $\varphi: \delta(G) \rightarrow M_2^0(\F_2) \rtimes \GL_2(\F_2)$. The possible orders of the image are:
\begin{equation}
    1, 2, 3, 4, 6, 8, 12, 16, 24, 48.
\end{equation}
Using {\tt Magma}, we check if $|\delta(G)| \ge 32$, there is no homomorphism from a problematic group to a subgroup of order $24$ or $48$ in the codomain of $\varphi$. Hence, either $|\delta(G)| \le 24$ or the image of $\varphi$ is $\le 16$. In either case, we can replace $\delta(G)$ by a quotient of order $\le 24$. Hence, \eqref{eq:mayle_wang_improvement_max_1} is improved to the same bound in \eqref{eq:mayle_wang_improvement_max_4}.

\begin{comment}
\begin{equation}
\label{eq:mayle_wang_improvement_max_2}
    \begin{split}
        p 
        & \leq (1.745(47\log\rad(2N_E N_{E'}) + 48\log(48)) + 48 \cdot 0 .23 + 6.8)^2 \\
        & \leq (83 \log\rad(2N_E N_{E'}) + 343)^2.
    \end{split}
\end{equation}

To complete the result, we consider the maximum of all possible cases together (that is, the maximum of \eqref{eq:mayle_wang_improvement_max_2} and \eqref{eq:mayle_wang_improvement_max_4}) which is exactly \eqref{eq:mayle_wang_improvement_result}.
\end{comment}
\end{proof}

\section{The results of Rouse and Zureick-Brown on $2$-adic images}

The $2$-adic representation $\rho_{E,2} : G_\Q \rightarrow \GL_2(\Z_2)$ of an elliptic curve $E$ over $\Q$ has open image in $\GL_2(\Z_2)$ and the properties that $\det \rho_{E,2} : G_ \Q \rightarrow \Z_2^\times$ is surjective and $\rho_{E,2}(c)$ is an element with determinant $-1$ and trace $0$ for a complex conjugation. With this in mind, the authors in \cite{Rouse-Zureick-Brown} make the following definition:
\begin{definition}
\label{arithmetic-max}
    An open subgroup $H \subseteq \GL_2(\Z_2)$ is arithmetically maximal if
    \begin{enumerate}
        \item $\det: H \rightarrow \Z_2^\times$ is surjective,
        \item there is an element of $H$ with determinant $-1$ and trace $0$, and
        \item the is no subgroup $K$ satisfying (1) and (2) with $H \subsetneq K$ and so that the genus of $X_K$ is $\ge 2$.
    \end{enumerate}
\end{definition}
The idea behind this definition is that arithmetically maximal subgroups $H \subseteq \GL_2(\Z_2)$ are maximal among the subgroups $H$ satisfying (1) and (2), except possibly when $H$ is contained in a subgroup $K$ such that $X_K$ has genus $\le 1$. For instance, if $H \subsetneq K$ and $X_K$ has genus $\ge 2$, it would be easier and sufficient to determine the $\Q$-rationals point $X_K$ rather than $X_H$.

For an arithmetically maximal subgroup $H$, either $X_H$ has infinitely many $\Q$-rational points (hence has genus $\le 1$) or $X_H$ has finitely many $\Q$-rational points. The union of the latter cases leads to a finite list of $j$-invariants. 

In \cite{Rouse-Zureick-Brown}, it is shown there are $727$ arithmetically maximal subgroups $H \subseteq \GL_2(\Z_2)$ up to conjugation such that $-I \in H$. There are an additional $1006$ arithmetically maximal subgroups up to conjugation such that $-I \notin H$. Thus, there are a total of $1733$ arithmetically maximal subgroups $H \subseteq \GL_2(\Z_2)$. Among these, there are $1414$ which have genus $\le 1$ and $1208$ of these are such that $X_H$ has infinitely many $\Q$-rational points.

In \cite[Theorem 1.1]{Rouse-Zureick-Brown}, the possible images of $\rho_{E,2}$ are determined in the following sense:
\begin{theorem}
\label{rouse-reformulate}
Let $H \subseteq \GL_2(\Z_2)$ be a subgroup and let $E/\Q$ be an elliptic curve such that the image of $\rho_{E,2}$ is contained in a conjugate of $H$. Then one of the following holds:
\begin{enumerate}
    \item The modular curve $X_H$ has infinitely many $\Q$-rational points (hence has genus $\le 1$).
    \item The elliptic curve $E$ has complex multiplication.
    \item The $j$-invariant of $E$ is one of
    \begin{align}
        \label{finite-j}
        & 2^{11}, 2^4 \cdot 17^3, 4097^3/2^4, 257^3/2^8, -857985^3/62^8, 919425^3/496^4, \\
        \notag & -3 \cdot 18249920^3/17^{16}, 7 \cdot 1723187806080^3/79^{16}.
    \end{align}
    \end{enumerate}
\end{theorem}

We rephrase the above theorem for the purposes of proving the main results of this paper.

\begin{theorem}
\label{images-2-adic}
Let $E/\Q$ be an elliptic curve without complex multiplication. Then $\rho_{E,2}(G)$ is one of the $1208$ arithmetically maximal groups listed in \cite{Rouse-Zureick-Brown} which have infinitely many $\Q$-rational points, or the $j$-invariant of $E$ appears in \eqref{finite-j}.
\end{theorem}

\section{Improved bounds for Serre's open image theorem}
In Theorem~\ref{thm:mayle_wang_improvement}, there is a hypothesis that $\rhobar_{E,2} \simeq \rhobar_{E',2}$ is absolutely irreducible. While we do not have an argument to remove this condition and achieve better bounds than Theorem~\ref{thm:mayle_wang_result}, we are able to do so in the case when $E$ and $E'$ are quadratic twists of each other and do not have complex multiplication.

\begin{proof}[Proof of Theorem~\ref{thm:mayle_wang_twist}]

Suppose $E'$ is a twist of $E$ by a quadratic character $\chi$ associated to the extension $\Q(\sqrt{d})$. Then $\rho_{E',2} = \rho_{E,2} \otimes \chi$. Since $E$ does not have complex multiplication, $E$ and $E'$ are not $\Q$-isogenous. 

Assume $\rho_{E,2}$ and hence $\rho_{E',2}$ are not absolutely irreducible.

If $\Q(\sqrt{d})$ is not a subfield of $L = \Q(E[2^\infty])$, then $L(\sqrt{d})$ is a degree $2$ extension of $L$. Hence, the identity automorphism of $L$ extends to an automorphism $\sigma$ of $L(\sqrt{d})$ such that $\chi(\sigma) = -1$. It follows that $\rho_{E,2}(\sigma) = I$ and $\rho_{E',2}(\sigma) = -I$ where $I$ is the identity element. This means that $\alpha = \beta = 1$ so we may apply the proofs of Theorem~\ref{thm:mayle_wang_improvement} and Proposition~\ref{prop:mayle_wang_varphi} (no need for Theorem~\ref{thm:carayol}) to get the desired conclusion.

Otherwise $\Q(\sqrt{d})$ lies in the field $L$. In \cite{Rouse-Zureick-Brown} (see Theorem~\ref{images-2-adic}), the possible $2$-adic images $\rho_{E,2}(G)$ of an elliptic curve $E/\Q$ are determined up to conjugacy. Every such image contains the principal congruence subgroup of level $32$ and can be regarded as a subgroup of $\GL_2(\Z/2^t\Z)$ for $0 \le t \le 5$. In order to apply \cite[Lemma 3.3]{Rouse-Zureick-Brown}, without loss of generality we take $t \ge 2$.

Let $\Gamma = 1 + 2^{t+1} M_2(\Z_2) \subseteq N = 1 + 2^t M_2(\Z_2)$. The character $\chi$ corresponds to a subgroup $H$ of index $2$ inside $\rho_{E,2}(G)$. Either $\Gamma \subseteq N \subseteq H$ or 
\begin{equation}
  N/N \cap H \cong NH/H \cong G/H
\end{equation}
so $N/N \cap H$ has order $2$. In the latter case, $N \cap H$ is a maximal subgroup of $N$, hence by \cite[Lemma 3.3]{Rouse-Zureick-Brown}, we obtain again that $\Gamma \subseteq N  \cap H \subseteq H$.

It follows that $\chi$ factors through
\begin{equation}
  \rho_{E,2}(G)/\Gamma.
\end{equation}

Consider the product representation
\begin{equation}
  \rho_2 = \rho_{E,2} \times \rho_{E',2} : G \rightarrow \GL_2(\Z_2) \times \GL_2(\Z_2).
\end{equation}
We note that 
\begin{equation}
    \rho_{E',2}(g) = \begin{cases}
        \rho_{E,2}(g) & \text{ if } \rho_{E,2}(g) \in H \\
        -\rho_{E,2}(g) & \text{ if } \rho_{E,2}(g) \notin H.
    \end{cases}
\end{equation}
Let $\theta$ be defined as:
\begin{align}
  \theta : \rho_{E,2}(G) & \rightarrow \rho_{E',2}(G) \\
\notag \rho_{E,2}(g) & \mapsto \begin{cases}
            \rho_{E,2}(g) & \text{ if } \rho_{E,2}(g) \in H \\
            -\rho_{E,2}(g) & \text{ if } \rho_{E,2}(g) \notin H.
        \end{cases}
\end{align}
The map $\theta$ is an isomorphism which is the identity when restricted to $\Gamma \subseteq H$ and 
\begin{equation}
  \rho_2(G) = (\rho_{E,2} \times \rho_{E',2})(G)
\end{equation}  
is the subgroup of $\GL_2(\Z_2) \times \GL_2(\Z_2)$ given by the graph of $\theta$, namely,
\begin{equation}
   \left\{ (\rho_{E,2}(g), \theta(\rho_{E,2}(g))) : g \in G\right\}.
\end{equation}
Hence, it follows that the inclusion
\begin{align}
    \iota: \rho_{E,2}(G) & \rightarrow \rho(G) \\
    \rho_{E,2}(g) & \mapsto (\rho_{E,2}(g), \theta(\rho_{E,2}(g)))
\end{align}
is an isomorphism. Composing with the homomorphism $\rho_2(G) \rightarrow (M/2M)^\times$, we obtain a homomorphism
\begin{equation}
    \phi: \rho_{E,2}(G) \rightarrow \rho_2(G) \surjects  \delta(G) \subseteq (M/2M)^\times.
\end{equation}

Since $\rho_{E,2}(G) \supseteq 1 + 2^t M_2(\Z_2)$, we see that $M = \Z_2[\rho_2(G)] \supseteq 2^t \Delta(M_2(\Z_2))$ where $\Delta: M_2(\Z_2) \rightarrow M_2(\Z_2) \times M_2(\Z_2)$ is the diagonal homomorphism. Thus, $2M \supseteq 2^{t+1} \Delta(M_2(\Z_2))$. Thus, we see that 
the kernel of the map $\phi$ contains $\Gamma = 1 + 2^{t+1} M_2(\Z_2)$ and hence $\phi$ factors through
\begin{equation}
  \rho_{E,2}(G)/\Gamma.
\end{equation}
Using {\tt Magma}, we check that each one of the $1208$ possibilities for $\rho_{E,2}(G)$ does not have a quotient which factors through $\rho_{E,2}(G)/\Gamma$ and is isomorphic to one of the problematic groups of order $192$, $128$, or $96$ in Table~\ref{problem-table}, if the image of $\rhobar_{E,2}$ is not absolutely irreducible. We conclude that it is possible to replace $\delta(G)$ with a quotient of order $\le 64$. It is also checked that the elliptic curves $E$ over $\Q$ with $j$-invariant in the finite list of Theorem~\ref{rouse-reformulate} (3) have $C_E \le 2$.

Hence, we obtain the bound
\begin{equation}
\label{eq:mayle_wang_improvement_max_3}
    \begin{split}
        p 
        & \leq (1.755(127\log\rad(2N_E N_{E'}) + 128\log(128)) + 128 \cdot 0.23 + 6.8)^2 \\
        & \leq (223 \log\rad(2N_E N_{E'}) + 1127)^2.
    \end{split}
\end{equation}

\end{proof}

\begin{proof}[Proof of Theorem \ref{thm:mayle_wang_surjective_result}]
We combine Theorems~\ref{thm:mayle_wang_improvement} and \ref{thm:mayle_wang_twist} with the arguments in \cite{Mayle-Wang}.
\end{proof}

\bibliographystyle{acm}
\bibliography{references}

\end{document}